\documentclass[12pt]{article}

 \usepackage{amsfonts,amssymb,amsthm,amsmath,graphicx,mathrsfs,color}

\usepackage[pdfstartview=FitH,
            CJKbookmarks=true,
            bookmarksnumbered=true,
            bookmarksopen=true,
            colorlinks,
            linkcolor=blue,
            anchorcolor=blue,
            citecolor=blue,
            urlcolor=blue
            ]{hyperref}

\newtheorem{theo}{Theorem}
\newtheorem{lemma}[theo]{Lemma}
\newtheorem{coro}[theo]{Corollary}
\newtheorem{prop}[theo]{Proposition}

\newtheorem{remark}[theo]{Remark}

\def\Z{\mathbb Z}
\def\Q{\mathbb Q}
\def\N{\mathbb N}
\def\R{\mathbb R}

\def\C{\mathbb C}
\def\Val{\overline{\mathrm{Val}}}

\begin{document}
\author{K\'aroly J. B\"or\"oczky, M\'aty\'as Domokos, Ansgar Freyer,
\\Christoph Haberl, Jin Li}

\title{Exponential valuations on lattice polygons valued at formal power series}


\maketitle

\begin{abstract}
We classify valuations on lattice polygons with values in the ring of formal power series that commute with the action of the affine unimodular group. A typical example of such valuations is induced by the Laplace transform, but as it turns out there are many more. The classification is done in terms of formal power series that satisfy certain functional equations.  
We align our classification with the decomposition into so-called dilative components.
\end{abstract}
\bigskip


\section{Introduction}

By a {\it polytope} $P$ in $\R^n$, we mean the convex hull of finitely many points of $\R^n$. $P$ is called a {\it segment}
if ${\rm dim}\,P=1$, and $P$ is called a {\it polygon}
if ${\rm dim}\,P=2$.
Let $\mathcal{F}$ be a family of compact convex sets in $\R^n$ such that if $P\cup Q$ is convex for $P,Q\in\mathcal{F}$,
then $P\cup Q\in\mathcal{F}$ and $P\cap Q\in\mathcal{F}$. Examples of such families are the
   family $\mathcal{K}^n$ of all convex compact sets in $\R^n$, the family
$\mathcal{P}^n$ of all polytopes in $\R^n$
and the family
$\mathcal{P}(\Z^n)$ of all lattice polytopes; namely, convex hulls of finitely many points of $\Z^n$
(see McMullen \cite{McM09}).
If $\mathcal{A}$ is a cancellative monoid (cancellative commutative semigroup with identity element $0_\mathcal{A}$), then a function $Z:\mathcal{F}\to \mathcal{A}$
is called a {\it valuation} if the following holds: if $P\cup Q$ is convex for $P,Q\in\mathcal{F}$, then
\begin{equation}
\label{valuationdef}
Z(P\cup Q)+Z(P\cap Q)=Z(P)+Z(Q).
\end{equation}
We say that the valuation $Z:\mathcal{F}\to \mathcal{A}$ is {\it simple} if $Z(P)=0_\mathcal{A}$ for any $P\in\mathcal{F}$ with ${\rm dim}\,P\leq n-1$. Typically, one would consider valuations intertwining with some natural group actions, as we will shortly see.

While the idea of valuations on convex polytopes played a crucial role in Dehn's solution of Hilbert's Third Problem already around 1900, after sporadic results, the systematic study of valuations only started with Hadwiger's celebrated characterization of the intrinsic volumes as the basis of the space of continuous isometry invariant valutions from 1957. For the  breathtaking developments of the last seven decades, see for example the monograph Alesker \cite{Ale18}, and the survey papers
Alesker \cite{Ale17}, Ludwig \cite{Lud23} and Ludwig, Mussnig \cite{LuM23}. 
The theory of valuations on lattice polytopes has been flourishing since the classical paper by Betke, Kneser \cite{BeK85} in 1985 characterizing unimodular invariant valuations on lattice polytopes, see for example, B\"or\"oczky, Ludwig \cite{BoL17,BoL19}, 
Jochemko, Sanyal \cite{JoS17,JoS18}, Berg, Jochemko, Silverstein \cite{BJM18} and Ludwig, Silverstein \cite{LuS17}.

To state the result stimulating our research, let $L^1_c(\R^n)$ denote the family
of Lebesgue integrable functions with compact support on $\R^n$. In addition, for measurable $\Omega\subset\R^n$, let $\mathcal{M}(\Omega)$ denote the family of Lebesgue measurable functions on $\Omega$. For $f\in L^1_c(\R^n)$, its Laplace transform is
$$
\mathcal{L}f(u)=\int_{\R^n}e^{-\langle u,v\rangle}f(v)\,dv.
$$
Here, $\langle u,v\rangle$ denotes the standard inner product on $\R^n$.
Li, Ma \cite{LiM17} applied the definition of the Laplace transform to  a  convex body $K$ by applying $\mathcal{L}$
to the characteristic function; namely,
$$
\mathcal{L}K=\mathcal{L}\mathbf{1}_K=\int_Ke^{-\langle u,v\rangle}\,dv.
$$

Inspired by the properties of the Laplace transform
applied to the characteristic functions of compact convex sets,
Li, Ma \cite{LiM17} considered valuations $Z\colon\mathcal{K}^n\to C(\R^n)$, where $C(\R^n)$ denotes the space of continuous functions on $\R^n$, which commute with affine transformations in the following sense: For any $K\in\mathcal{K}^n$ and any $\Phi\in{\rm GL}(n,\R)$, $w\in\R^n$ we have
\begin{equation}
    \label{eq:equivariance_negative_sign}
    Z(\Phi K + w)(u) = e^{-\langle u,w\rangle} |\det\Phi| \cdot Z(K)(\Phi^Tu)
\end{equation}

\begin{theo}[Li, Ma \cite{LiM17}]
Let $Z\colon \mathcal{K}^n \to C(\R^n)$ be a continuous (with respect to the Hausdorff metric) valuation satisfying \eqref{eq:equivariance_negative_sign}. Then there exists a constant $c\in\R$ such that $Z = c\mathcal L$.
\end{theo}

In this paper, we consider valuations  on lattice polygons (of $\Z^2$) valued at formal power series in two variables.
We consider the action of the affine unimodular group (the group of affine transformation leaving $\Z^2$ invariant) 
\[
\mathcal{G}(\Z^2) = \Z^2 \rtimes {\rm GL}(2,\Z)
\]
on $\Q[[x,y]]$ (the ring of formal power series in the variables $x$ and $y$) where for $f\in \Q[[x,y]]$ and $\Xi\in\mathcal G(\Z^2)$ with
\[
\Xi(x,y) = (ax+by + \alpha,~cx+dy+\beta),
\]
$a,b,c,d,\alpha,\beta\in\Z$ and $ad-bc=\pm 1$, we have
\begin{equation}
\label{AffineActionOnPowerSeries}
(\Xi\cdot f)(x,y)=\exp (\alpha x+\beta y)\cdot f(ax+cy,bx+dy).
\end{equation}
In this context the exponential function is considered as a formal power series
$\exp(t)=\sum_{n=0}^\infty\frac{t^n}{n!}\in\Q[[t]]$, which still satisfies  that
$\exp(t+s)=\exp(t)\cdot \exp(s)$ for any formal power series $s,t$ (cf. Sambale \cite{Sam23}). 

On lattice polygons, $\mathcal{G}(\Z^2)$ acts in the natural way. Now a valuation  $Z:\mathcal{P}(\Z^2)\to \Q[[x,y]]$ is called $\mathcal{G}(\Z^2)$ equivariant, if $Z(\Xi P) = \Xi\cdot Z(P)$ holds for all $P\in\mathcal P(\Z^2)$ and $\Xi\in\mathcal G(\Z^2)$. This equivariance property is essentially analogue to the one in \eqref{eq:equivariance_negative_sign} except for a sign in the exponential function.

One example of a $\mathcal G(\Z^2)$ equivariant valuation is the ``positive Laplace transform" $\mathcal{L}_+$
defined for $P\in\mathcal P(\Z^2)$ as
$$
\mathcal{L}_+(P)(x,y)=\int_{\R^2}\exp(\alpha x+\beta y)\cdot \mathbf{1}_P(\alpha,\beta)\,d\alpha d\beta.
$$
This valuation has been studied (in general dimension) by Barvinok \cite{Ba91,Ba08} and Lawrence \cite{L91a,L91b}, which has lead to spectacular progress in computational geometry.

The main goal of this paper is to characterize $\mathcal G(\Z^2)$ equivariant 
 valuations $Z:\mathcal{P}(\Z^2)\to \Q[[x,y]]$  in a constructive way, providing a tool to produce such valuations on the one hand, and to be able to decide whether a valuation  on $\mathcal{P}(\Z^2)\to\Q[[x,y]]$ is $\mathcal G(\Z^2)$ equivariant on the other hand. Our method is partially based on ideas in 
B\"or\"oczky, Domokos, Freyer, Haberl, Harcos,  Li \cite{BDFHHL}.
We note that Freyer, Ludwig, Rubey \cite{FMR} characterized $\mathcal{G}(\mathbb{Z}^2)$ equivariant valuations valued at the formal power series in two variables, but their approach is less constructive (see below for a summary of their result). 

We write $e_1,e_2$ to denote the orthonormal basis of $\R^2$ also generating $\Z^2$,
and set $T=[e_1,e_2,o]$ where $o=(0,0)$ stands for the origin, and $[x_1,\ldots,x_k]$ stands for the convex hull of
$x_1,\ldots,x_k\in \R^2$. In the formulas below,
$$
\frac{e^t-1}{t}\mbox{ is identified with the formal power series }\sum_{n=0}^\infty\frac{t^n}{(n+1)!}. 
$$
In addition, we write $D_4$ to denote the subgroup of $\mathrm{GL}(2,\Z)$ generated by the matrices 
$\left[\begin{array}{cc}1 & 0 \\1 & -1\end{array}\right]$ and 
$\left[\begin{array}{cc}1 & -2 \\0 & -1\end{array}\right]$. 
The reason for the chosen notation is that this group  $D_4$ is isomorphic to the dihedral group of $8$ elements. 
The subalgebra of $\Q[[x,y]]$ of $D_4$ invariant elements is
(see Remark~\ref{D4-natural-basis})
\begin{align}
\nonumber
\Q[[x,y]]^{D_4}=&\{h\in \Q[[x,y]]:\,\Phi\cdot h=h\mbox{ for }\Phi\in D_4\}\\
\label{D4-natural-basis-intro}
=& \{g(2x^2+2xy+y^2, 4x^2y^2+4xy^3+y^4)\in \Q[[x,y]]:\\
\nonumber
&\mbox{ }g(a,b)\in \Q[[a,b]]\}
\end{align}
where the polynomials
$2x^2+2xy+y^2$ and $4x^2y^2+4xy^3+y^4$ are algebraically independent.

\begin{theo}
\label{PowerSeriesValGL2Z}
For any $\mathcal G(\Z^2)$ equivariant
valuation $Z:\mathcal{P}(\Z^2)\to \Q[[x,y]]$, 
$Z(\{o\})$
 is a constant   power series $c \in \Q$, and
\begin{equation}
\label{f1PowerSeries}
f_1(x,y):=Z([o,e_1])(x,y)=
g(x^2)\cdot \exp\left(\mbox{$\frac{1}2$}\,x\right)
\end{equation}
holds for some $g\in \Q[[x]]$; moreover, 
$Z$ satisfies 
\begin{equation}
\label{ZZ2Z1}
Z(T)(x,y)=f_2(x,y)+\mbox{$\frac12$}\,f_1(x,y)+
\mbox{$\frac12$}\,f_1(y,-x)
+\mbox{$\frac{e^{x}}2$}\,f_1(-x+y,-x)
\end{equation}
where $f_2$ is defined by 
\begin{equation}
\label{simpleZrho}
f_2(x,y)=\frac{e^x}{y}\cdot \frac{e^{y-x}-1}{y-x}\cdot \varrho(y-x,x)-
\frac{1}y\cdot \frac{e^{x}-1}{x}\cdot \varrho(x,y-x)
\end{equation}
 for a $\varrho\in\Q[[x,y]]^{D_4}$
satisfying
\begin{equation}
\label{rhoformula}
(2x+y) \varrho(x,y)=(x+y)\varrho(x,x+y)+x \varrho(x+y,x).
\end{equation}
In particular, $Z=Z_1+Z_2$ where $Z_1$ (constructed in Proposition \ref{lowdim}) and $Z_2$ are $\mathcal G(\Z^2)$ equivariant
valuations, and $Z_2$ is simple with $Z_2(T)=f_2$.

On the other hand, if $c\in\Q$, 
$f_1(x,y)=g(x^2)\cdot \exp\left(\mbox{$\frac{1}2$}\,x\right)$
for some $g\in \Q[[x]]$,
and $f_2$ is defined by \eqref{simpleZrho} 
for a
$\varrho\in\Q[[x,y]]^{D_4}$ 
satisfying \eqref{rhoformula},
then there exists a $\mathcal G(\Z^2)$ equivariant
valuation $Z:\mathcal{P}(\Z^2)\to \Q[[x,y]]$  such that $Z(\{o\})\equiv c$, 
$Z([o,e_1])=f_1$, and
$Z(T)$ is defined by \eqref{ZZ2Z1}.
\end{theo}

In particular, simple $\mathcal G(\Z^2)$ equivariant
valuations $Z:\mathcal{P}(\Z^2)\to \Q[[x,y]]$ are in bijective correspondance with formal power series $\varrho\in\Q[[x,y]]^{D_4}$ 
satisfying \eqref{rhoformula} via the formula (cf. \eqref{simpleZrho})
\begin{equation}
\label{simpleZrho-PowerSeries0}
Z(T)(x,y)=\frac{e^x}{y}\cdot \frac{e^{y-x}-1}{y-x}\cdot \varrho(y-x,x)-
\frac{1}y\cdot \frac{e^{x}-1}{x}\cdot \varrho(x,y-x)
\end{equation}
where the right hand side is a formal power series under these conditions. Here $\varrho$ is even (it contains only terms of even degree) by \eqref{D4-natural-basis-intro}, and actually, $\varrho$ does not contain terms of degree $2$ (cf.\ Corollary \ref{coro:vd_dims}).

In order to 
characterize a
$\mathcal G(\Z^2)$ equivariant
valuation $Z:\mathcal{P}(\Z^2)\to \Q[[x,y]]$, Freyer, Ludwig, Rubey \cite{FMR} say that $Z$ is $\delta$-dilative  for $\delta\in\Z$ if for any (two-dimensional) lattice polygon $P$, we have
\begin{equation}
\label{d-dilative}
Z(mP)(x,y)=m^{-\delta}\cdot Z(P)(mx,my)
\end{equation}
for any integer $m\geq 1$. As an example, the positive Laplace transform $\mathcal L_+$ is (-2)-dilative.
It turned out that the dilative valuations give a natural decomposition to the space of $\mathcal{G}(\Z^2)$ valuations $\mathcal P(\Z^2) \to \Q[[x,y]]$. As in \cite{FMR}, we write $\Val$ for the vector space of $\mathcal{G}(\Z^2)$ equivariant valuations $\mathcal P(\Z^2) \to \Q[[x,y]]$, and moreover, $\Val_\delta$ for the subspace of $\delta$-dilative valuations.

\begin{theo}[Freyer, Ludwig, Rubey \cite{FMR}]
\label{theorem:dilative_decomp}
    We have \[\Val = \prod_{\delta\geq -2} \Val_\delta,\] as well as, \[\{Z\in\Val \colon Z\textrm{ simple}\} = \prod_{\delta \geq -2 \text{ even},~\delta\neq 0} \Val_\delta.\]
    Moreover,
        \[
            \dim\Val_\delta = \begin{cases}
                0, &\delta < -2,\\
                1, &\delta > -2\text{ odd},\\
                \lfloor \tfrac{\delta+2}{12}\rfloor + 1, &\delta\geq -2\text{ even and }\,\delta\,{\rm mod}\,12 \neq 0,\\
                \lfloor \tfrac{\delta+2}{12}\rfloor,&\delta\geq -2\text{ even and }\,\delta\,{\rm mod}\,12 = 0.
            \end{cases}
        \]
\end{theo}

It should be noted that in \cite{FMR} power series with real coefficients are considered. However, the choice of coefficients does not play a crucial role in these results as the bases of the vector spaces involved may be chosen to be rational. Hence, Theorem \ref{theorem:dilative_decomp} holds true for rational coefficients as well.




Given the above theorem, it is natural to ask for a characterization of $\delta$-dilative valuations in terms of the data $c\in\Q$, $g\in\Q[[x]]$ and $\rho\in\mathcal V$ from Theorem \ref{PowerSeriesValGL2Z}, where $\mathcal V$ denotes the space of formal power series $\varrho\in\Q[[x,y]]^{D_4}$ 
satisfying \eqref{rhoformula}. We write $\mathcal V_d$ for the space of homogeneous polynomials of degree $d$ in $\mathcal V$. 

\begin{theo}
\label{thm:dilative_parameters}
    Let $\delta \geq -2$ be an integer and let $Z\colon\mathcal P(\Z^2) \to \Q[[x,y]]$ be a $\mathcal{G}(\Z^2)$ equivariant valuation parametrized by $c\in\Q$, $g\in\Q[[x]]$ and $\rho\in \mathcal V$ as in Theorem \ref{PowerSeriesValGL2Z}.
    \begin{enumerate}
        \item Let $\delta$ be even and non-zero. Then, $Z$ is $\delta$-dilative if and only if $c=0$, $g=0$ and $\rho\in \mathcal V_{\delta+2}$.
        \item Let $\delta$ be odd. Then $Z$ is $\delta$-dilative if and only if $c=0$, $g(x)= \alpha x^{\delta/2}\sinh(\tfrac{\sqrt x}{2})$ for some $\alpha \in \Q$, and $\rho=0$.
        \item $Z$ is $0$-dilative if and only if $c=\alpha$, $g(x)=\alpha \cosh(\tfrac{\sqrt x}{2})$ for some $\alpha\in \Q$, and $\rho=0$.
    \end{enumerate}
\end{theo}

\begin{remark}
Theorem~\ref{PowerSeriesValGL2Z} and Theorem~\ref{thm:dilative_parameters} hold with real or complex coefficients, as well, instead of rational coefficients (actually, any field of characteristic zero can be chosen instead of $\Q$).

In view of Theorem \ref{theorem:dilative_decomp} and Theorem~\ref{thm:dilative_parameters}, for even $d\geq 0$, the dimension of $\mathcal V_{d}$  (and thus of $\Val_{d-2}$ ) agrees with the dimension of the linear space of modular forms of weight $d$ with respect to the full ${\rm SL}(2,\Z)$. The latter space carries a natural product structure, while $\mathcal V$ is not closed under the usual multiplication in $\Q[[x,y]]$. It is not clear whether there exists a natural product on the space of simple $\mathcal G(\Z^2)$ equivariant valuations, or at least on the vector space $\mathcal{V}$.
\end{remark}

Concerning the structure of the paper, results due to B\"or\"oczky, Domokos, Freyer, Haberl, Harcos,  Li \cite{BDFHHL} needed for our paper are discussed in Section~\ref{secGroupAction}.
Theorem~\ref{PowerSeriesValGL2Z} is proved in
Section~\ref{sec-rho-polynomial}, and finally, Theorem~\ref{thm:dilative_parameters} is verified in Section~\ref{sec-Dimension-Modular-Forms}.

\section{Preliminaries}
\label{secGroupAction}

In this section, we summarize some related results obtained by B\"or\"oczky, Domokos, Freyer, Haberl, Harcos,  Li \cite{BDFHHL}. These results were obtained for $\mathcal G(\Z^2)$ equivariant valuations that are valued at \emph{measurable functions} in two variables, where from an algebraic perspective, the action of the affine unimodular group on modular function is equivalent to the definition here in the case of power series in two variables; it is also described by the formula  \eqref{AffineActionOnPowerSeries}. The results quoted here are not specific to measurable functions and hold, which is why we can replace the target space by $\Q[[x,y]]$.


According to Proposition~8 in \cite{BDFHHL},   we have the following:

\begin{prop}
\label{ExpValSL2Z}
For  any two $\mathcal{G}(\Z^2)$ equivariant
valuations $Z,Z':\mathcal{P}(\Z^2)\to \Q[[x,y]]$, the following statements hold:
\begin{description}
\item{(i)} If $Z(\{o\})=Z'(\{o\})$ and $Z([o,e_1])=Z'([o,e_1])$, then
$Z-Z'$ is a simple $\mathcal{G}(\Z^2)$ equivariant
valuation;
\item{(ii)}
If $Z(\{o\})=Z'(\{o\})$, $Z([o,e_1])=Z'([o,e_1])$ and $Z(T)=Z'(T)$,
then $Z=Z'$.
\end{description}
\end{prop}

It follows from Proposition~\ref{ExpValSL2Z} that
in order to characterize a $\mathcal{G}(\Z^2)$ equivariant valuation $Z:\mathcal{P}(\Z^2)\to \Q[[x,y]]$,
all we need to characterize are
\begin{equation}
\label{f0f1f2}
Z(\{o\})=f_0\mbox{ and }
Z([o,e_1])=f_1\mbox{ and }
Z(T)=f_2.
\end{equation}
The basic algebraic identities for $f_0$ and $f_1$ are described by Lemma~9 in \cite{BDFHHL}.

\begin{lemma}
\label{f0f1f2properties}
Let
$Z:\mathcal{P}(\Z^2)\to \Q[[x,y]]$ be a $\mathcal{G}(\Z^2)$ equivariant
valuation. Then
$Z(\{o\})=f_0$ and $Z([o,e_1])=f_1$ satisfy the following properties:
\begin{align}
\label{f0GL2ZSL2Z0}
f_0(ax+cy,bx+dy)&=f_0(x,y) \mbox{ \ \ \ for $\left[
\begin{array}{cc}
a&b\\
c&d
\end{array}\right]\in {\rm GL}(2,\Z)$};\\
\label{f1shift}
f_1(-x,-y)&= \exp(-x)f_1(x,y);\\
\label{f1111}
f_1(x,y)&= f_1(x,x+y);\\
\label{f1extra}
f_1(x,y)&= f_1(x,-y).
\end{align}
\end{lemma}

In turn, Proposition~11 in \cite{BDFHHL} says that the conditions \eqref{f0GL2ZSL2Z0}, \eqref{f1shift}, \eqref{f1111} and
 \eqref{f1extra} are sufficient.

\begin{prop}
\label{lowdim}
For $f_0,f_1\in\Q[[x,y]]$ satisfying  \eqref{f0GL2ZSL2Z0}, \eqref{f1shift}, \eqref{f1111} and
 \eqref{f1extra}, there exists
a  $\mathcal{G}(\Z^2)$ equivariant valuation $Z_1:\mathcal{P}(\Z^2)\to\Q[[x,y]]$ satisfying
$Z_1(\{o\})=f_0$ and $Z_1([o,e_1])=f_1$ and
$$
Z_1(T)(x,y)=\mbox{$\frac12$}\,f_1(x,y)+\mbox{$\frac12$}\,f_1(y,-x)
+\mbox{$\frac12$}\,\exp(x)\cdot f_1(-x+y,-x).
$$
\end{prop}

Proposition~\ref{ExpValSL2Z} (i) and Proposition~\ref{lowdim} yield the following.

\begin{coro}
\label{ReductionToSimple}
For any  $\mathcal{G}(\Z^2)$ equivariant
valuation $Z:\mathcal{P}(\Z^2)\to \Q[[x,y]]$, let
$Z_1$ be the $\mathcal{G}(\Z^2)$ equivariant
valuation satisfying $Z_1(\{o\})=Z(\{o\})$ and
$Z_1([o,e_1])=Z([o,e_1])$
constructed in Proposition~\ref{lowdim}. Then $Z_2=Z-Z_1$
is a $\mathcal{G}(\Z^2)$ equivariant simple
valuation.
\end{coro}

Our next goal is to understand the algebraic properties of simple $\mathcal{G}(\Z^2)$ equivariant
 valuation $Z:\mathcal{P}(\Z^2)\to \Q[[x,y]]$. The functional equations for $Z(T)=f_2$ are described by  Lemma~13 in  \cite{BDFHHL}.

\begin{lemma}
\label{simplef2algebraic-properties}
For any simple $\mathcal{G}(\Z^2)$ equivariant valuation $Z:\mathcal{P}(\Z^2)\to \Q[[x,y]]$,
$Z(T)=f_2$ satisfies the following properties: 
\begin{align}
\label{f2simple1}
f_2(-x+y,-x)&= \exp(-x)\cdot f_2(x,y);\\
\label{f2simple2}
f_2(x,y)+\exp(x+y)\cdot f_2(-x,-y)&=f_2(x,x+y)+f_2(x+y,y);\\
\label{f2simple3}
f_2(x,y)&= f_2(y,x).
\end{align}
\end{lemma}

Proposition~14 in  \cite{BDFHHL} states the reverse statement.

\begin{prop}
\label{simplef2algebraic}
For any $f_2\in\Q[[x,y]]$ satisfying
the properties \eqref{f2simple1}, \eqref{f2simple2} and \eqref{f2simple3}
in Lemma~\ref{simplef2algebraic-properties},
there exists a unique simple $\mathcal{G}(\Z^2)$ equivariant
valuation $Z:\mathcal{P}(\Z^2)\to \Q[[x,y]]$ such that
$Z(T)=f_2$.
\end{prop}

Finally, Lemma~21 in  \cite{BDFHHL} observes that in Lemma~\ref{simplef2algebraic-properties} and Proposition~\ref{simplef2algebraic}, we can exchange \eqref{f2simple2} with \eqref{f23up}.

\begin{lemma}
\label{f23uplemma}
 Assuming that $f_2\in\Q[[x,y]]$ satisfies \eqref{f2simple1} and \eqref{f2simple3}, we have $f_2$ also satisfies
\eqref{f2simple2} if and only if 
\begin{equation}
\label{f23up}
f_2(x,y)+\exp(x)\cdot f_2(y-x,y)=f_2(x,x+y)+f_2(y,x+y).
\end{equation}
\end{lemma}

\section{Parametrization of equivariant valuations}
\label{sec-rho-polynomial}

In this section, we combine the tools summarized in Section~\ref{secGroupAction} with new results in order to prove Theorem~\ref{PowerSeriesValGL2Z}. We note that
according to Sambale \cite{Sam23}, for any formal power series of the form
$\varphi(x)=1+\sum_{n=1}^\infty b_nx^n\in\Q[[x]]$, there exists a formal power series $\psi\in \Q[[x]]$ such that
\begin{equation}
\label{logarithm-PowerSeries}
\exp(\psi)=\varphi.
\end{equation}

The following statement is well-known in invariant theory.

\begin{lemma}
\label{fyx+y}
If $f\in\Q[[x,y]]$ satisfies $f(x,y)=f(x,x+y)$, then there exists $\tilde{f}\in \Q[[x]]$
such that $f(x,y)=\tilde{f}(x)$.
\end{lemma}

 It follows that if both $f(x,y)=f(x,x+y)$ and $f(x,y)=f(x+y,y)$ hold for $f\in\Q[[x,y]]$, then $f$ is a constant power series.



Next we show that all we need to understand are the 
simple 
$\mathcal G(\Z^2)$ equivariant
valuations.

\begin{prop}
\label{PowerSeriesValGL2ZZ1}
For any $\mathcal G(\Z^2)$ equivariant
valuation $Z:\mathcal{P}(\Z^2)\to \Q[[x,y]]$, 
$Z(\{o\})$
 is a constant power series, and
\begin{equation}
f_1(x,y)=Z([o,e_1])(x,y)=
g(x^2)\cdot \exp\left(\mbox{$\frac{1}2$}\,x\right)
\end{equation}
holds for some $g\in \Q[[x]]$; moreover, 
there exists a simple 
$\mathcal G(\Z^2)$ equivariant
valuation $Z_2:\mathcal{P}(\Z^2)\to \Q[[x,y]]$
such that
\begin{equation}
\label{ZZ2Z1-PowerSeries}
Z(T)(x,y)=Z_2(T)(x,y)+\mbox{$\frac12$}\,f_1(x,y)+
\mbox{$\frac12$}\,f_1(y,-x)
+\mbox{$\frac12$}\,\exp(x)\cdot f_1(-x+y,-x).
\end{equation}

On the other hand, if $c\in\Q$ and 
$f_1(x,y)=g(x^2)\cdot \exp\left(\mbox{$\frac{1}2$}\,x\right)$
for some $g\in \Q[[x]]$,
then there exists a  $\mathcal G(\Z^2)$ equivariant
valuation $Z:\mathcal{P}(\Z^2)\to \Q[[x,y]]$  such that $Z(\{o\})\equiv c$ and 
$Z([o,e_1])=f_1$.
\end{prop}
\begin{proof}
For a $\mathcal G(\Z^2)$ equivariant
valuation $Z:\mathcal{P}(\Z^2)\to \Z[[x,y]]$,
let $Z(\{o\})=f_0$  and 
$Z([o,e_1])=f_1$.
According to Lemma~\ref{f0f1f2properties},
$f_0$ and $f_1$ satisfy the following properties:
\begin{eqnarray}
\label{f0GL2ZSL2Z0-PowerSeries}
f_0(x,y)&=&f_0(ax+cy,bx+dy) \mbox{ \ for any $
\left[
\begin{array}{cc}
a&b\\
c&d
\end{array}\right]\in {\rm GL}(2,\Z)$};\\
\label{f1shift-PowerSeries}
f_1(x,y)&=& \exp(x)\cdot f_1(-x,-y);\\
\label{f1111-PowerSeries}
f_1(x,y)&=& f_1(x,x+y);\\
\label{f1extra-PowerSeries}
f_1(x,y)&=& f_1(x,-y).
\end{eqnarray}
It follows from 
\eqref{f0GL2ZSL2Z0-PowerSeries} and Lemma~\ref{fyx+y} that $f_0$ is a constant power series.

For $f_1$, \eqref{f1111-PowerSeries} and Lemma~\ref{fyx+y} yield that $f_1(x,y)=\tilde{f}(x)$ for a
$\tilde{f}\in \Q[[x]]$, and by \eqref{f1shift-PowerSeries} we have
\begin{equation}
\label{f1tildef-shift-PowerSeries}
\tilde{f}(x)=\exp(x)\cdot \tilde{f}(-x).
\end{equation}
We may assume that $\tilde{f}$ is not the constant zero power series,
and let $a_dx^d$, $d\in\N$, be the term with the smallest degree with non-zero coefficient $a_d$ in $\tilde{f}(x)$. 
It follows that
$\tilde{f}(x)=a_dx^d\cdot \varphi(x)$ where
$\varphi(x)\in\Q[[x]]$ is of the form 
$\varphi(x)=1+\sum_{n=1}^\infty b_nx^n$. Substituting 
$\tilde{f}(x)=a_dx^d\cdot \varphi(x)$ into \eqref{f1tildef-shift-PowerSeries} and equating the coefficients of $x^d$ on the two sides shows that $d=2k$ for a $k\in\N$, and hence
\begin{equation}
\label{f1tildef-shift-PowerSeries-phi}
\varphi(x)=\exp(x)\cdot \varphi(-x).
\end{equation}
According to \eqref{logarithm-PowerSeries}, there exists a formal power series $\psi\in \Q[[x]]$ such that
$\exp(\psi)=\varphi$; therefore, \eqref{f1tildef-shift-PowerSeries-phi} yields that $\psi(x)=x+\psi(-x)$. We deduce the existence of an $h\in\Q[[x]]$ such that
$\psi(x)=\frac12\,x+h(x^2)$, and hence $f_1(x,y)=
g(x^2)\cdot \exp\left(\mbox{$\frac{1}2$}\,x\right)$ where $g(x)=a_{k}x^k\cdot\exp(h(x))$.

Now let us assume that 
$f_0,f_1\in\Q[[x,y]]$ satisfy that $f_0\equiv c$ for a constant $c\in\Q$ and $f_1(x,y)=
g(x^2)\cdot \exp\left(\mbox{$\frac{1}2$}\,x\right)$ for a $g\in\Q[[x]]$. In particular, $f_0$ and $f_1$ satisfy
\eqref{f0GL2ZSL2Z0-PowerSeries}, \eqref{f1shift-PowerSeries}, \eqref{f1111-PowerSeries} and \eqref{f1extra-PowerSeries}.
It follows from Proposition~\ref{lowdim} that
 there exists
a  $\mathcal{G}(\Z^2)$ equivariant valuation $Z_1:\mathcal{P}(\Z^2)\to\Q[[x,y]]$ satisfying
$Z_1(\{o\})=f_0$ and $Z_1([o,e_1])=f_1$ and
$$
Z_1(T)(x,y)=\mbox{$\frac12$}\,f_1(x,y)+\mbox{$\frac12$}\,f_1(y,-x)
+\mbox{$\frac12$}\,\exp(x)\cdot f_1(-x+y,-x).
$$
\end{proof}

Given Proposition~\ref{PowerSeriesValGL2ZZ1}, our remaining task is to understand simple $\mathcal G(\Z^2)$ equivariant
valuations; therefore, the rest of the section is dedicated to simple valuations.

 The following statement follows from Lemma~\ref{simplef2algebraic-properties}, Proposition~\ref{simplef2algebraic} and Lemma~\ref{f23uplemma}.
 
\begin{prop}
\label{simplef2algebraic-PowerSeries}
For any simple
$\mathcal G(\Z^2)$ equivariant
valuation $Z:\mathcal{P}(\Z^2)\to \Q[[x,y]]$, 
$Z(T)=f_2$ satisfies the  properties
\begin{description}
\item{(A)} $f_2(x,y)+\exp(x)\cdot f_2(y-x,y)=f_2(x,x+y)+f_2(y,x+y)$;
\item{(B)} $f_2(x,y)= f_2(y,x)$;
\item{(C)} $f_2(y-x,-x)= \exp(-x)\cdot f_2(x,y)$.
\end{description}

On the other hand, for any $f_2\in\Q[[x,y]]$ satisfying
the properties (A), (B) and (C),
there exists a unique simple
$\mathcal G(\Z^2)$ equivariant
valuation $Z:\mathcal{P}(\Z^2)\to \Q[[x,y]]$ such that
$Z(T)=f_2$.
\end{prop}

In the formulas below, the expression
$\frac{e^t-1}{t}$ stands for $\sum_{n=0}^\infty\frac{t^n}{(n+1)!}\in\Q[[t]]$, and its reciprocal $\frac{t}{e^t-1}$
stands for
$\sum_{n=0}^\infty\frac{B_n}{n!}\cdot t^n\in\Q[[x]]$
where $B_0, B_1,\ldots$
are the Bernoulli numbers  (cf. Zagier \cite{Zag08}).
In particular, we can also consider
$\frac{e^y-e^x}{y-x}=\exp(x)\cdot \frac{e^{y-x}-1}{y-x}\in \Q[[x,y]]$.
Tacitly, we also use the property that $\Q[[x,y]]$ is an integral domain; namely, the product of non-zero elements is non-zero
(cf. Sambale \cite{Sam23}).

The main idea to better understand simple power series valued valuations is to transform $f_2$ into a power series
$\varrho$ that satisfies some relations (see (A') and (E) in Theorem~\ref{ABC-AprimeE-PowerSeries})
  that do not contain exponential components and are simpler than
the conditions (A), (B) and (C) for $f_2$ in Proposition~\ref{simplef2algebraic-PowerSeries}.
In particular, for $f\in\Q[[x,y]]$, we consider the formal power series
\begin{equation}
\label{rhotildef-PowerSeries}
f^\sharp(x,y)=
\frac{x}{e^x-1}\cdot
\frac{x+y}{e^{x+y}-1}\cdot \Big[f(x,x+y)+\exp(x)\cdot f(y,x+y)\Big].
\end{equation}
As a reverse notion, if $\varrho\in\Q[[x,y]]$ and 
$\frac{e^{y}-e^x}{y-x}\cdot \varrho(y-x,x)-
 \frac{e^{x}-1}{x}\cdot \varrho(x,y-x)$ is of the form $y\cdot g(x,y)$ for a 
 $g\in\Q[[x,y]]$, then
we consider
\begin{equation}
\label{tildefrho-PowerSeries}
\varrho^\dag(x,y)=
\frac{1}y\cdot\left[ \frac{e^{y}-e^x}{y-x}\cdot \varrho(y-x,x)-
 \frac{e^{x}-1}{x}\cdot \varrho(x,y-x)\right].
\end{equation}
We observe that $\varrho^\dag(x,y)$ can't be defined, for example, for $\varrho(x,y)=x$; however,
 $\varrho^\dag$ is a formal power series if either
 $\varrho=f^\sharp$ for an $f\in\Q[[x,y]]$ (cf. Lemma~\ref{f-AAprime-PowerSeries}), or if $\varrho$ satisfies (cf. Lemma~\ref{even-degree-terms})
\begin{description}
\item{(D)} $\varrho(-x,-y)=\varrho(x,y)$
\end{description}
where (D) is equivalent to saying that each term in $\varrho$ with non-zero coefficient has even degree. 
We note that the operators
$f\mapsto f^\sharp$ and
$\varrho\mapsto \varrho^\dag$ are essentially inverses of each other
 (cf. Lemma~\ref{f-AAprime-PowerSeries}).

 The first major step of our argument towards understanding simple valuations is the following statement:

 \begin{theo}
\label{ABC-AprimeE-PowerSeries}
If $f\in\Q[[x,y]]$ satisfies
(A), (B) and (C) in Proposition~\ref{simplef2algebraic-PowerSeries}, then
$\varrho=f^\sharp$ satisfies $\varrho^\dag=f$ and
\begin{description}
\item{(A')} $(x+y) \varrho(x,y-x)=y\varrho(x,y)+x \varrho(y,x)$ and
\item{(E)} $\varrho(x,-2x-y)=\varrho(x,y)$.
\end{description}

On the other hand, if 
$\varrho\in\Q[[x,y]]$ satisfies conditions (A') and (E), then $\varrho$ satisfies (D), and hence
$\varrho^\dag\in\Q[[x,y]]$, and $f=\varrho^\dag$ satisfies
(A), (B) and (C).
\end{theo}

A good part of Theorem~\ref{ABC-AprimeE-PowerSeries} is proved directly  in 
Lemma~\ref{f-AAprime-PowerSeries}.

\begin{lemma}
\label{f-AAprime-PowerSeries}
Let $f\in\Q[[x,y]]$.
\begin{description}
\item{(i)} 
 $\varrho^\dag=f$ for $\varrho=f^\sharp$.

\item{(ii)} If $\varrho^\dag=f$ for $\varrho\in\Q[[x,y]]$, then
  $\varrho=f^\sharp$.

\item{(iii)} 
If $f\in\Q[[x,y]]$ satisfies 
(A) in Proposition~\ref{simplef2algebraic-PowerSeries}, then 
$\varrho=f^\sharp$ satisfies
(A') in Theorem~\ref{ABC-AprimeE-PowerSeries}.

On the other hand, 
if $f=\varrho^\dag$ for a $\varrho\in\Q[[x,y]]$ that satisfies (A'), then
$f$ satisfies (A).
\end{description}
\end{lemma}
\begin{proof}
For (i), we observe that
if $\varrho=f^\sharp$, then
$$
\frac{e^{y}-e^x}{y-x}\cdot \varrho(y-x,x)-
 \frac{e^{x}-1}{x}\cdot \varrho(x,y-x)=y\,f(x,y)
$$
follows from substituting the formula for
$\varrho$ in terms of $f$ coming from
\eqref{rhotildef-PowerSeries} into the left hand side and using the basic rules for the exponential power series.

For (ii), if $f=\varrho^\dag\in\Q[[x]]$, then the argument is similar,  
only we substitute the formula for $f$ in terms of $\varrho$
\eqref{tildefrho-PowerSeries} into the right hand side of \eqref{rhotildef-PowerSeries}.

To verify (iii), we may assume that $\varrho=f^\sharp$ by (i) and (ii), and then
using \eqref{rhotildef-PowerSeries}, we deduce that 
\begin{align*}
&(x+y) \varrho(x,y-x)-y\varrho(x,y)-x \varrho(y,x)\\
=&\frac{xy(x+y)}{(e^x-1)(e^y-1)}\cdot\left[f(x,y)+\exp(x)\cdot f(y-x,y)\right]\\
&-\frac{xy(x+y)}{(e^x-1)(e^{x+y}-1)}\cdot\left[f(x,x+y)+\exp(x)\cdot f(y,x+y)\right]\\
&-\frac{xy(x+y)}{(e^y-1)(e^{x+y}-1)}\cdot\left[f(y,x+y)+\exp(y)\cdot f(x,x+y)\right]\\
=&\frac{x}{e^x-1}\cdot\frac{y}{e^y-1}\cdot (x+y)\times\\
&\times\left[f(x,y)+\exp(x)\cdot f(y-x,y)-f(x,x+y)-f(y,x+y)\right]
\end{align*}
where the elements of the field of fractions of $\Q[[x,y]]$ appearing in the above formulae 
belong actually to $\Q[[x,y]]$.  
\end{proof}

Given Lemma~\ref{f-AAprime-PowerSeries}, we may assume that $\varrho\in \Q[[x,y]]$ satisfies (A').

\begin{lemma}
\label{lemma:rho(x,y-x)}
If (A') holds for $\varrho\in \Q[[x,y]]$, then  
\begin{align}
\rho(x,y-x)&=\rho(y,x-y) \label{eq:rho(x,y-x)}
\\ 
\rho(y,-x)&=\rho(y-x,x)  \label{eq:rho(y,-x)}
\\ 
\rho(x,y)&=\rho(x+y,-y) \label{eq:rho(x,y)=rho(x+y,-y)}
\end{align} 
\end{lemma} 

\begin{proof} The right hand side of (A') is symmetric in $x$ and $y$. Therefore if 
(A') holds for $\varrho$, then $(x+y)\rho(x,y-x)$ is symmetric in $x$ and $y$, 
implying in turn that $\rho(x,y-x)$ is symmetric in $x$ and $y$. 
That is, we have  \eqref{eq:rho(x,y-x)}.

The equality \eqref{eq:rho(y,-x)} 
is obtained by making the linear substitution  
$x\mapsto y$, $y\mapsto y-x$ in \eqref{eq:rho(x,y-x)}, 
whereas \eqref{eq:rho(x,y)=rho(x+y,-y)} is obtained by making the substitution 
$x\mapsto x$, $y\mapsto x+y$ in \eqref{eq:rho(x,y-x)}. 
\end{proof} 

Now we show that condition 
(D) for a $\varrho\in\Q[[x,y]]$
ensures that
$\varrho^\dag \in\Q[[x,y]]$.

\begin{lemma}
\label{even-degree-terms}
For $\varrho\in\Q[[x,y]]$,
$\varrho^\dag$ is a formal power series if $\varrho\in\Q[[x,y]]$ satisfies 
$\varrho(-x,-y)=\varrho(x,y)$ (condition (D) preceding Theorem~\ref{ABC-AprimeE-PowerSeries}). 
\end{lemma}
\begin{proof} 
Assume that 
the power series $\varrho(x,y)=\sum_{p,q\geq 0}a_{p,q}x^py^q$ satisfies condition (D). 
Then $p+q$ is even for each term $a_{p,q}x^py^q$ with $a_{p,q}\neq 0$. 
In particular, we have 
\begin{equation*}
\label{rhominusxx}
\varrho(-x,x)-\varrho(x,-x)=0.
\end{equation*}
It follows that $h(x,y)=\frac{e^{y}-e^x}{y-x}\cdot \varrho(y-x,x)-
 \frac{e^{x}-1}{x}\cdot \varrho(x,y-x)$ 
satisfies 
\[h(x,0)=\frac{e^{x}-1}{x}\cdot (\varrho(-x,x)-\varrho(x,-x))=0.\]  
Therefore  
$h(x,y)=y\cdot g(x,y)$ for a formal power series 
$g(x,y)\in \Q[[x,y]]$, showing 
that $\varrho^\dag(x,y)=\frac1y\cdot h(x,y)\in\Q[[x,y]]$.
\end{proof}

\begin{lemma}
\label{ABC-ABCprime-PowerSeries}
For $\varrho\in\Q[[x,y]]$,
$\varrho$ satisfies the conditions
(A') and (E) in  Theorem~\ref{ABC-AprimeE-PowerSeries} if and only if it satisfies the conditions
(A'), (B') and (C') where
\begin{description}
\item{(B')} $(x-y)\varrho(x,y-x)=x\varrho(y,-x)-y\varrho(x,-y)$ and
\item{(C')} $(x-y)\varrho(-x,x-y)=x\varrho(y,-x)-y\varrho(x,-y)$.
\end{description}

In addition, the
conditions (B') and (C') for $\varrho$ yield (D) in Lemma~\ref{even-degree-terms}.
\end{lemma}
\begin{proof}
Let us quickly show that the conditions (B') and (C') for $\varrho\in\Q[[x,y]]$ yield (D). Since the right hand sides of (B') and (C') coincide, we have
$\varrho(x,y-x)=\varrho(-x,x-y)$. Thus the invertible change of variable $(x,y)\mapsto (x,x+y)$ implies  that $\varrho(x,y)=\varrho(-x,-y)$. 

First, let $\varrho\in\Q[[x,y]]$
satisfy (A'), (B') and (C'), and hence also (D). 
Substituting $y$ by $-y$ in (B') leads to 
$$
(x+y)\cdot\varrho(x,-y-x)=x\cdot\varrho(-y,-x)+y\cdot \varrho(x,y).
$$
The right hand side above equals the right hand side of (A') by (D), hence 
$(x+y)\cdot\varrho(x,-y-x)=(x+y)\cdot\varrho(x,y-x)$ by (A'), which in turn implies that 
\begin{equation}
\label{eq:D'} 
\varrho(x,y-x)=\varrho(x,-y-x). 
\end{equation} 
Substituting $y$ by $x+y$ in \eqref{eq:D'}, we obtain (E).  

We assume now that (A') and (E) hold for $\varrho$. 
Using the substitution $y\mapsto y-x$ in (E) shows that
\eqref{eq:D'} holds, as well.
It follows from Lemma~\ref{lemma:rho(x,y-x)} that $\varrho$ also satisfies
\eqref{eq:rho(x,y)=rho(x+y,-y)}, and so we have
both 
$$
\varrho(x,y)=\varrho(x+y,-y) \mbox{ and } \varrho(x,y)=\varrho(x,-2x-y).
$$
Composing the above two linear substitutions, we get  
$$
\varrho(x,y)=\varrho(x+y,-2(x+y)-(-y))=\varrho(x+y,-2x-y).
$$
In addition, composing the linear substitution $x\mapsto x+y$, $y\mapsto -2x-y$ with itself, we obtain that 
$$
\varrho(x,y)=\varrho((x+y)+(-2x-y),-2(x+y)-(-2x-y))=\varrho(-x,-y),
$$
and so (D) holds. 
Now make the substitution $y\mapsto -y$ in (A') to get 
\[(x-y)\rho(x,-y-x)=x\rho(-y,x)-y\rho(x,-y).\]
By (D), we can replace $\rho(-y,x)$ by $\rho(y,-x)$ on the right hand side of the above equality, and we can replace $\rho(x,-y-x)$ by $\rho(x,y-x)$ on the left hand side of the above equality by \eqref{eq:D'}. This way we obtain (B'). 

Finally, the change of variable $y\mapsto y-x$ in (D) yields that $\varrho(x,y-x)=\varrho(-x,x-y)$, and hence (B') and (D) imply (C'), 
completing the proof of Lemma~\ref{ABC-ABCprime-PowerSeries}.
\end{proof}

\begin{proof}[Proof of Theorem~\ref{ABC-AprimeE-PowerSeries}]
According to Lemma~\ref{f-AAprime-PowerSeries}, 
Lemma~\ref{even-degree-terms} and Lemma~\ref{ABC-ABCprime-PowerSeries}, it is equivalent to prove that
$f\in\Q[[x,y]]$ satisfies (A),
(B) and (C) in Proposition~\ref{simplef2algebraic-PowerSeries} if and only if
$\varrho=f^\sharp \in \Q[[x,y]]$ satisfies (A'), (B') and (C') in Lemma~\ref{ABC-ABCprime-PowerSeries}.

According to
Lemma~\ref{f-AAprime-PowerSeries} and
Lemma~\ref{even-degree-terms},
we may assume that
$f\in\Q[[x,y]]$ satisfies (A),
$\varrho=f^\sharp$ satisfies (A') and $f=\varrho^\dag$. We deduce from Lemma~\ref{lemma:rho(x,y-x)} that
$\varrho$ satisfies
\eqref{eq:rho(x,y-x)} and
\eqref{eq:rho(y,-x)}, as well.

Under these conditions, the condition (B) for $f$ is equivalent to (B') for $\varrho$, as using \eqref{eq:rho(x,y-x)} and \eqref{eq:rho(y,-x)}, we have
\begin{align*} 
&f(x,y)-f(y,x)=\varrho^\dag(x,y)-\varrho^\dag(y,x)=\\
=&\frac{1}{xy(y-x)}\cdot
\Big[(e^y-e^x)x\varrho(y-x,x)
+(e^x-1)(x-y)\varrho(x,y-x)
\\ 
&+(e^x-e^y)y\rho(x-y,y)+(e^y-1)(y-x)\varrho(y,x-y)\Big]\\ 
=&\frac{e^x-e^y}{xy(y-x)}\cdot \Big[(x-y)\varrho(x,y-x)-x\varrho(y-x,x)+y\varrho(x-y,y)\Big]\\ =&\frac{e^x-e^y}{xy(y-x)}\cdot \Big[(x-y)\varrho(x,y-x)-x\varrho(y,-x)+y\varrho(x,-y)\Big], 
\end{align*}
where the calculations show that the power series in brackets always have the required divisibility properties. Therefore, we may also assume that
$f$ satisfies (B) and $\varrho$ satisfies (B').  Under these additional conditions, we verify that $\varrho^\dag=\varrho^\Diamond$ where 
$$
\varrho^\Diamond(x,y)=\frac{1}{x-y}\cdot \left[\frac{e^x-1}{x}\cdot\varrho (x,-y)-\frac{e^y-1}{y}\cdot\varrho(y,-x)\right]
$$
where the power series in the brackets is anti-symmetric, and hence divisible by $y-x$.

Now as $\varrho$ satisfies (B') and \eqref{eq:rho(y,-x)}, the formula $\varrho^\dag=\varrho^\Diamond$ follows from
\begin{align*}
\varrho^\dag-\varrho^\Diamond=& 
\frac{e^y-e^x}{y(y-x)}\cdot\varrho(y-x,x)-\frac{e^x-1}{xy}\cdot\varrho(x,y-x)-\\ 
&-\frac{1}{x-y}\cdot\left[\frac{e^x-1}{x}\cdot\varrho(x,-y)-\frac{e^y-1}{y}\cdot\varrho(y,-x)\right]
\\ 
=&\frac{1}{xy(y-x)}\cdot\Big[(e^y-e^x)\cdot x\cdot\varrho(y,-x)+(e^x-1)(x-y)\cdot\varrho(x,y-x)
\\ 
& +(e^x-1)\cdot y\cdot \varrho(x,-y)-(e^y-1)\cdot x\cdot \varrho(y,-x)\Big]\\ 
=&\frac{e^x-1}{xy(x-y)}\cdot\Big[(x-y)\cdot\varrho(x,y-x)-x\cdot \varrho(y,-x)+y\cdot\varrho(x,-y)\Big].
\end{align*}
In particlar, we have $f=\varrho^\Diamond$.

Finally, the condition (C) for $f$ is equivalent to (C') for $\varrho$, as using \eqref{eq:rho(y,-x)}, we have
\begin{align*} 
&f(y-x,-x)- \exp(-x)\cdot f(x,y)=
\varrho^\Diamond(y-x,-x)-e^{-x}\varrho^\Diamond(x,y)=\\
=& \frac{1}{y}\cdot\left[\frac{e^{y-x}-1}{y-x}\cdot\varrho(y-x,x)-\frac{e^{-x}-1}{-x}\cdot\varrho(-x,x-y)\right]
\\ 
& \quad \quad -\frac{e^{-x}}{x-y}\cdot\left[\frac{e^x-1}{x}\cdot\varrho(x,-y)-\frac{e^y-1}{y}\cdot\varrho(y,-x)\right]
\\ 
=& \frac{e^{-x}}{xy(x-y)}\cdot\Big[x(e^x-e^y)\cdot\varrho(y,-x)-(x-y)(e^x-1)\cdot\varrho(-x,x-y)-
\\  
&  \quad\quad\quad \quad\quad \quad -y(e^x-1)\cdot\varrho(x,-y)+x(e^y-1)\cdot\varrho(y,-x)\Big]
\\ 
=& \frac{e^{-x}(e^x-1)}{xy(x-y)}\cdot\Big[x\cdot\varrho(y,-x)-(x-y)\cdot\varrho(-x,x-y)-y\cdot\varrho(x,-y)\Big],  
\end{align*}
completing the proof of Theorem~\ref{ABC-AprimeE-PowerSeries}.
\end{proof}

We write $D_4$ to denote the subgroup of $\mathrm{GL}(2,\Z)\subset \mathrm{GL}(2,\C)$ generated by the matrices 
$\left[\begin{array}{cc}1 & 0 \\1 & -1\end{array}\right]$ and 
$\left[\begin{array}{cc}1 & -2 \\0 & -1\end{array}\right]$. 
In particular, this group  $D_4$ is isomorphic to the dihedral group of $8$ elements.
We recall (cf. \eqref{AffineActionOnPowerSeries}) that a
$\Phi=\left[
\begin{array}{cc}
a&b\\
c&d
\end{array}\right]\in {\rm GL}(2,\Z)$ acts on a $h\in \Q[[x,y]]$ in a way such that
\begin{equation}
\label{GL2ZActionOnPowerSeries}
(\Phi\cdot h)(x,y)= h(ax+cy,bx+dy).
\end{equation}
Now the subalgebra of $\Q[[x,y]]$ of $D_4$ invariant elements is
(see Remark~\ref{D4-natural-basis})
\begin{align}
\nonumber
\Q[[x,y]]^{D_4}=&\{h\in \Q[[x,y]]:\,\Phi\cdot h=h\mbox{ for }\Phi\in D_4\}\\
\label{D4invariant-PowerSeries}
=&\{g(2x^2+2xy+y^2,4x^2y^2+4xy^3+y^4)\in \Q[[x,y]]:\\
\nonumber
&\mbox{ }g(a,b)\in \Q[[a,b]]\}
\end{align}
where the the polynomials
$2x^2+2xy+y^2$ and $4x^2y^2+4xy^3+y^4$ are algebraically independent.

\begin{proof}[Proof of Theorem~\ref{PowerSeriesValGL2Z}] 
For a $\mathcal{G}(\Z^2)$ equivariant
valuation $Z:\mathcal{P}(\Z^2)\to \Q[[x,y]]$, it is sufficient to characterize 
$Z(\{o\})=f_0$ and 
$Z([o,e_1])=f_1$  and 
$Z(T)$ according to Proposition~\ref{ExpValSL2Z}.

It follows from Proposition~\ref{PowerSeriesValGL2ZZ1}
that $Z(\{o\})=f_0$
 is a constant power series, and
$$
f_1(x,y)=Z([o,e_1])(x,y)=
g(x^2)\cdot \exp\left(\mbox{$\frac{1}2$}\,x\right)
$$
holds for some $g\in \Q[[x]]$; moreover, 
there exists a simple 
$\mathcal G(\Z^2)$ equivariant
valuation $Z_2:\mathcal{P}(\Z^2)\to \Q[[x,y]]$
such that
$$
Z(T)(x,y)=Z_2(T)(x,y)+\mbox{$\frac12$}\,f_1(x,y)+
\mbox{$\frac12$}\,f_1(y,-x)
+\mbox{$\frac12$}\,\exp(x)\cdot f_1(-x+y,-x).
$$

On the other hand, Proposition~\ref{PowerSeriesValGL2ZZ1} also says that if $c\in\Q$ and 
$f_1(x,y)=g(x^2)\cdot \exp\left(\mbox{$\frac{1}2$}\,x\right)$
for some $g\in \Q[[x]]$,
then there exists a  $\mathcal G(\Z^2)$ equivariant
valuation $Z:\mathcal{P}(\Z^2)\to \Q[[x,y]]$  such that $Z(\{o\})\equiv c$ and 
$Z([o,e_1])=f_1$.

Therefore, the remaining task is to characterize $f_2=Z_2(T)$ for a simple $\mathcal G(\Z^2)$ equivariant
valuation $Z_2:\mathcal{P}(\Z^2)\to \Q[[x,y]]$.

It follows from Proposition~\ref{simplef2algebraic-PowerSeries} and Theorem~\ref{ABC-AprimeE-PowerSeries} that
$f_2=Z_2(T)$ for a simple $\mathcal G(\Z^2)$ equivariant
valuation $Z_2:\mathcal{P}(\Z^2)\to \Q[[x,y]]$
if and only if $f=\varrho^\dag$ for a $\varrho\in\Q[[x,y]]$ satisfying conditions (A') and (E) in Theorem~\ref{ABC-AprimeE-PowerSeries}.

In turn, Theorem~\ref{PowerSeriesValGL2Z} follows from the statement that
$\varrho\in\Q[[x,y]]$ satisfies the conditions (A') and (E) in Theorem~\ref{ABC-AprimeE-PowerSeries} if and only if $\varrho\in \Q[[x,y]]^{D_4}$ and (A') holds for
$\varrho$.

Therefore, let $\varrho\in\Q[[x,y]]$ satisfy the condition (A'), and hence
\eqref{eq:rho(x,y)=rho(x+y,-y)} holds for
$\varrho$ by Lemma~\ref{lemma:rho(x,y-x)}.
Since \eqref{eq:rho(x,y)=rho(x+y,-y)} and (E) states that $\varrho$ is invariant under 
$\left[\begin{array}{cc}1 & 0 \\1 & -1\end{array}\right]$ and 
$\left[\begin{array}{cc}1 & -2 \\0 & -1\end{array}\right]$,
and (A') and \eqref{rhoformula} are readily equivalent,
we conclude Theorem~\ref{PowerSeriesValGL2Z}.    
\end{proof}

\begin{remark}
\label{D4-natural-basis}
We obtain a more traditional representation of $D_4$ after the change of variable $s:=2x+y$, $t:=y$. 
In addition, the substitution $s:=2x+y$, $t:=y$ transforms the matrix group 
$D_4$ to the group (denoted also by $D_4$) generated by  
$\left[\begin{array}{cc}1 & 0 \\ 0 & -1\end{array}\right]$ and 
$\left[\begin{array}{cc} 0 & -1 \\ -1 & 0\end{array}\right]$, 
and then (see for example Hunziker \cite[Section 4]{hunziker}) 
$$
\Q[[s,t]]^{D_4}=
\{g(s^2+t^2,s^2t^2):\,g(a,b)\in \Q[[a,b]]\}
$$
where the the polynomials
$s^2+t^2$ and $s^2t^2$ are algebraically independent.
In terms of the original variables $x,y$, we deduce \eqref{D4invariant-PowerSeries}.

To understand the space of formal power series 
$\varrho\in \Q[[x,y]]^{D_4}$ satisfying condition (A') 
occurring in Theorem~\ref{PowerSeriesValGL2Z}, for $\varrho\in\Q[[x,y]]$, we consider the $\sigma\in\Q[[s,t]]$ defined by 
$\sigma(s,t)=\rho(x,y)$. 
Then $\rho(x,y)=\sigma(2x+y,y)$, hence $\rho(x,y-x)=\sigma(2x+y-x,y-x)=
\sigma(x+y,y-x)$, and $\rho(y,x)=\sigma(x+2y,x)$. 
Therefore, equation (A')
for $\varrho\in\Q[[x,y]]$ translates to the functional equation 
\begin{equation}\label{eq:(A'')} 
(s+t)\sigma(s+t,t-s)=s\sigma(s+2t,s)+t\sigma(2s+t,t)  
\end{equation}
for $\sigma\in\Q[[s,t]]$.
In particular, for any  $d>0$, the $\Q$-vector space $\mathcal{V}_d$ of $d$-homogeneous polynomials
$\varrho\in\Q[[x,y]]^{D_4}$ satisfying (A')
is isomorphic to the
$\Q$-vector space of $d$-homogeneous polynomials
$\sigma\in\Q[[s,t]]^{D_4}$ satisfying \eqref{eq:(A'')}.
\end{remark}

\section{Dilative Valuations}
\label{sec-Dimension-Modular-Forms}

In this section we prove Theorem \ref{thm:dilative_parameters}. We start with the first case, where $\delta\geq -2$ is an even and non-zero integer, which is the most involved case. Here, Theorem \ref{theorem:dilative_decomp} says that any $\delta$-dilative valuation is simple. 

 We recall (see \eqref{d-dilative} and the following remarks) that a simple 
$\mathcal G(\Z^2)$ equivariant
valuation $Z:\mathcal{P}(\Z^2)\to \Q[[x,y]]$ is $\delta$-dilative  for $\delta\in\Z$ if
\begin{equation}
\label{d-dilative0}
Z(mT)(x,y)=m^{-\delta}\cdot Z(T)(mx,my)
\end{equation}
for any integer $m\geq 2$.
In order to describe the left hand side of \eqref{d-dilative0}, if $m\in\N$, then we consider
$g_m\in \Q[[x,y]]$ defined by
\begin{equation}
\label{gmxy}
g_m(x,y)=
\sum_{t,s\in\N,\; t+s\leq m} \exp(sx+ty)=\sum_{(s,t)\in (mT)\cap\Z^2}\exp(sx+ty).
\end{equation}
In Lemma \ref{ZmT-PowerSeries} (ii), we use our usual short hand notation for  power series involving the exponential series.

\begin{lemma}
\label{ZmT-PowerSeries}
Let $Z:\mathcal{P}(\Z^2)\to \Q[[x,y]]$ be a $\mathcal{G}(\Z^2)$ equivariant
valuation, and let $f=Z(T)$.
\begin{description}
\item[(i)] If $m\geq 2$ for $m\in\N$, then
$$ 
Z(mT)(x,y)=g_{m-1}(x,y)\cdot f(x,y)+\exp(x+y)\cdot g_{m-2}(x,y)\cdot f(-x,-y).
$$ 
\item[(ii)] If $m\in\N$, then
$$
g_{m}(x,y)=
\frac{e^{x+y}\left(e^{(m+1)x}-e^{(m+1)y}\right)-\left(e^{(m+2)x}-e^{(m+2)y}\right)+e^x-e^y}{(e^x-e^y)(e^x-1)(e^y-1)}.
$$
\end{description}
\end{lemma}
\begin{proof}
The formula in (i) follows from the tiling of $mT$ by translates of $T$ and $-T$.

For (ii), we recall that $e_1=(1,0)$ and $e_2=(0,1)$, and  observe that the points of $\Z^2$ in $[(m+1)T]\backslash(e_i+mT)$ are of the form $\ell e_j$ where $\{i,j\}=\{1,2\}$ and $\ell=0,\ldots,m+1$. It follows that
$$
\exp(x)\cdot g_{m}(x,y)+\sum_{\ell=0}^{m+1}\exp(\ell y)=g_{m+1}(x,y)=
\exp(y)\cdot g_{m}(x,y)+\sum_{\ell=0}^{m+1}\exp(\ell x),
$$
which in turn yields (ii).
\end{proof}


\begin{prop}
\label{degreedpol-to-dilatived-2}
Let $\varrho\in\Q[[x,y]]^{D_4}$  satisfy the condition (A') in Theorem~\ref{ABC-AprimeE-PowerSeries}, and let 
$Z(T)=f=\varrho^\dag$ for the simple
$\mathcal G(\Z^2)$ equivariant
valuation 
$Z:\mathcal{P}(\Z^2)\to \Q[[x,y]]$ provided by
Proposition~\ref{simplef2algebraic-PowerSeries} and
Theorem~\ref{ABC-AprimeE-PowerSeries}.
If $\varrho$ is a homogeneous polynomial of degree $d$ for $d\in 2\N$, then $Z$ is $d-2$ dilative.
\end{prop}
\begin{proof} According to
 \eqref{simpleZrho-PowerSeries0} and Theorem~\ref{ABC-AprimeE-PowerSeries}, $f=Z(T)$ satisfies that
\begin{equation}
\label{simpleZrho-PowerSeries00}
f(x,y)= \frac{e^{x}-e^{y}}{(x-y)y}\cdot \varrho(y-x,x)-
 \frac{e^{x}-1}{xy}\cdot \varrho(x,y-x)
\end{equation}
Since $\varrho(-x,-y)=\varrho(x,y)$, we also have 
\begin{align}
\nonumber
f(-x,-y)&= \frac{e^{-x}-e^{-y}}{(x-y)y}\cdot \varrho(y-x,x)-
 \frac{e^{-x}-1}{xy}\cdot \varrho(x,y-x)\\
 \label{simpleZrho-PowerSeries00-minus}
&=-\frac{e^{x}-e^{y}}{e^{x+y}(x-y)y}\cdot \varrho(y-x,x)+
\frac{e^{x}-1}{e^x xy}\cdot \varrho(x,y-x).
\end{align}
In order to show that $Z$ is $(d-2)$-dilative, we calculate the left hand side of \eqref{d-dilative0} using
Lemma \ref{ZmT-PowerSeries}, 
\eqref{simpleZrho-PowerSeries00} and
\eqref{simpleZrho-PowerSeries00-minus}, and obtain that
$$
 Z(mT)(x,y)=G(x,y,m)\cdot \varrho(y-x,x)-
H(x,y,m)\cdot \varrho(x,y-x) 
$$
where
\begin{align*}
 G(&x,y,m)=
\frac{e^{x+y}\left(e^{mx}-e^{my}\right)-\left(e^{(m+1)x}-e^{(m+1)y}\right)+e^x-e^y}{(e^x-e^y)(e^x-1)(e^y-1)}
\cdot \frac{e^{x}-e^{y}}{(x-y)y}\\
&-
e^{x+y}\cdot\frac{e^{x+y}\left(e^{(m-1)x}-e^{(m-1)y}\right)-\left(e^{mx}-e^{my}\right)+e^x-e^y}{(e^x-e^y)(e^x-1)(e^y-1)}\cdot \frac{e^{x}-e^{y}}{e^{x+y}(x-y)y}\\
  &\qquad\quad=\frac{e^{mx}-e^{my}}{(x-y)y}
\end{align*}
and
\begin{align*}
 H(x,&y,m)=
\frac{e^{x+y}\left(e^{mx}-e^{my}\right)-\left(e^{(m+1)x}-e^{(m+1)y}\right)+e^x-e^y}{(e^x-e^y)(e^x-1)(e^y-1)}
\cdot \frac{e^{x}-1}{xy}\\
&-
e^{x+y}\cdot\frac{e^{x+y}\left(e^{(m-1)x}-e^{(m-1)y}\right)-\left(e^{mx}-e^{my}\right)+e^x-e^y}{(e^x-e^y)(e^x-1)(e^y-1)}\cdot \frac{e^{x}-1}{e^{x}xy}\\
&\qquad =\frac{e^{mx}-1}{xy}.
\end{align*}
Since $\varrho$ is homogeneous of degree $d$, we deduce that
$$
Z(mT)(x,y)=\frac{e^{mx}-e^{my}}{(x-y)y}\cdot \varrho(y-x,x)-
\frac{e^{mx}-1}{xy}\cdot \varrho(x,y-x)=
m^{2-d}f(mx,my),
$$
proving that $Z$ is $(d-2)$-dilative.
\end{proof}
In order to verify the converse of Proposition~\ref{degreedpol-to-dilatived-2}, we summarize basic properties of 
 $\delta$-dilative valuations
 based on Freyer, Ludwig, Rubey \cite{FMR}, and introduce various related notions. For $r\in\N$, we write $\Q[x,y]_r$ to denote the $(r+1)$-dimensional $\Q$ vector space of homogeneous polynomials of degree $r$, and for
$f\in \Q[[x,y]]$, we write $\pi_rf\in \Q[x,y]_r$ to denote the $r$th degree term in $f$; namely, $f$ is the formal sum $\sum_{r\in\N}\pi_rf$.

For $r\in\N$, we say that a valuation $\Psi:\mathcal{P}(\Z^2)\to \Q[x,y]_r$ is {\it translatively polynomial} if there exist functions $\Psi_{(i)}:\mathcal{P}(\Z^2)\to\Q[x,y]_i$ for $i=0,\ldots,r$ such that
if $v=(\alpha,\beta)\in\Z^2$ and $P\in\mathcal{P}(\Z^2)$, then
\begin{equation}
\label{trans-poly}
\Psi(P+v)(x,y)=\sum_{j=0}^r\Psi_{(r-j)}(P)(x,y)\cdot (\alpha x+\beta y)^j;
\end{equation}
or in the more compact tensor notation, $\Psi(P+v)=\sum_{j=0}^r\Psi_{(r-j)}(P)\otimes v^j$. In this case,
each $\Psi_{(i)}$ in \eqref{trans-poly} is a translatively polynomial valuation.  The crucial fact proved originally by McMullen \cite{McM77} if $r\leq 1$ and by Khovanskii, Pukhlikov \cite{KoP92} if $r\in\N$ is that 
there exists a $j$-homogeneous
translatively polynomial valuation $\Psi_j:\mathcal{P}(\Z^2)\to \Q[x,y]_r$
for $j=0,\ldots,r+2$ such that 
\begin{equation}
\label{trans-poly-decomp}
\Psi=\sum_{j=0}^{r+2}\Psi_j.
\end{equation}
It was observed by Freyer, Ludwig, Rubey in \cite[Section 6]{FMR} that a valuation $Z\colon\mathcal{P}(\Z^2) \to \Q[[x,y]]$ is $\mathcal{G}(\Z^2)$ equivariant, if and only if each summand $\pi_rZ \colon \mathcal{P}(\Z^2) \to \Q[x,y]_r$ is translatively polynomial and $\mathrm{GL}(2,\Z)$ equivariant, i.e.,
\[
\pi_rZ\left(\begin{pmatrix}
    a & b\\
    c & d
\end{pmatrix} P\right)(x,y) =\pi_rZ(P)(ax+cy,bx+dy),
\]
holds for all $a,b,c,d\in\Z$ with $|ad-bc|=1$.


For $r\in\N$, Freyer, Ludwig, Rubey \cite{FMR} consider the $\Q$ vector space ${\rm Val}^r$ of 
translatively polynomial  and ${\rm GL}(2,\Z)$ equivariant
valuations 
$Z:\mathcal{P}(\Z^2)\to \Q[x,y]_r$, and hence
\begin{equation}
\label{Valr-directsum}
{\rm Val}^r=\oplus_{j=0}^{r+2}\,
{\rm Val}^r_j
\end{equation}
where ${\rm Val}^r_j$ is the $\Q$ vector space ${\rm Val}^r$ of $j$-homogeneous
translatively polynomial  and ${\rm GL}(2,\Z)$ equivariant
valuations $Z:\mathcal{P}(\Z^2)\to \Q[x,y]_r$ for $j=0,\ldots,r$ (cf. \eqref{trans-poly-decomp}). 
In \cite[Lemma 25]{FMR} it is shown that a $\mathcal{G}(\Z^2)$ equivariant
valuation $Z:\mathcal{P}(\Z^2)\to \Q[[x,y]]$ is $\delta$-dilative  for an integer $\delta\geq -2$ if and only if
for each $r\in\N$, we have \begin{equation}
    \label{d-dilative-Val}
    \pi_r Z\in {\rm Val}^r_{r-\delta}.
\end{equation}

\begin{proof}[Proof of Theorem~\ref{thm:dilative_parameters}, Case 1] 
Recall that $\mathcal{V}\subset\Q[[x,y]]$ denotes the vector space of power series in $x$ and $y$ that satisfy \eqref{rhoformula} and that $\mathcal{V}_d$ is its degree $d$ part, i.e., $\mathcal
V_d = \pi_d \mathcal{V}$.


According to Theorem~\ref{PowerSeriesValGL2Z}, there exists a vector space isomorphism $\Phi$ from $\mathcal{V}$ to the $\Q$ vector space of all simple
$\mathcal G(\Z^2)$ equivariant
valuations 
$Z:\mathcal{P}(\Z^2)\to \Q[[x,y]]$ where
\begin{equation}
\label{simpleZrho-PowerSeries000}
Z(T)(x,y)= \frac{e^{x}-e^{y}}{(x-y)y}\cdot \varrho(y-x,x)-
 \frac{e^{x}-1}{xy}\cdot \varrho(x,y-x)
\end{equation}
 for $\varrho\in \cal{V}$ and $Z=\Phi(\varrho)$. We deduce from Proposition~\ref{degreedpol-to-dilatived-2} that if
$\varrho\in \mathcal{V}_d$ for some $d\in 2\N$, then
$\Phi(\varrho)$ is $(d-2)$-dilative, and hence \eqref{Valr-directsum} and \eqref{d-dilative-Val} yield that  then
\begin{equation}
\label{d-dilative-Val0}
\begin{array}{rcll}
\pi_r \Phi(\varrho)&=&0&\mbox{ if }r<d-2\\
\pi_r \Phi(\varrho)&\in& {\rm Val}^r_{r-d+2}&\mbox{ if }r\geq d-2.
\end{array}
\end{equation}

Now let $Z:\mathcal{P}(\Z^2)\to \Q[[x,y]]$ be a non-trivial $\delta$-dilative 
$\mathcal{G}(\Z^2)$ equivariant
valuation, where $\delta \geq -2$ is a non-zero integer. Then $Z$ is simple (cf.\ Theorem \ref{theorem:dilative_decomp}). In particular, 
$Z=\Phi(\varrho)$ for a $\varrho\in \cal{V}$, and let
$\varrho_d=\pi_d\varrho$ for $d\in\N$.
We suppose that there exists a $\bar{d}\in 2\N$ with $\bar{d}\neq \delta+2$ such that
$\varrho_{\bar{d}}\neq 0$, and seek a contradiction.
As $\Phi$ is an isomorphism and $\Phi(\rho_{\overline d})$ is $(\overline d -2)$-dilative, \eqref{d-dilative-Val} yields the existence of a $\bar{r}\geq \bar{d}-2$ such that
\begin{equation}
\label{indirect-dilative}
\pi_{\bar{r}} \Phi(\varrho_{\bar{d}})\in {\rm Val}^{\bar{r}}_{\bar{r}-\bar{d}+2}\backslash\{0\}.
\end{equation}
According to \eqref{Valr-directsum}, there exists a natural linear projection map
$\Pi:{\rm Val}^{\bar{r}}\to {\rm Val}^{\bar{r}}_{\bar{r}-\bar{d}+2}$, which then satisfies by \eqref{d-dilative-Val0}, \eqref{indirect-dilative} and $\bar{d}\neq \delta+2$ that
\begin{equation}
\label{good-bad-Pi}
\Pi\circ \pi_{\bar{r}}\circ \Phi(\varrho_{\bar{d}})\neq 0\mbox{ \ and \ }\Pi\circ \pi_{\bar{r}}\circ \Phi(\varrho)= 0.
\end{equation}
Let $\theta=\varrho-\sum_{d=0}^{\bar{r}+2}\varrho_d$, and hence $\theta$ may contain only non-trivial terms of degree at least  $\bar{r}+3$. It follows from \eqref{simpleZrho-PowerSeries000} that $\Phi(\theta)$ may contain only non-trivial terms of degree at least  $\bar{r}+1$; therefore, $\pi_{\bar{r}} \Phi(\theta)= 0$. 
We deduce from 
$\bar{r}+2\geq \bar{d}$,
\eqref{d-dilative-Val0} and \eqref{good-bad-Pi} that
$$
0=\Pi\circ \pi_{\bar{r}}\circ \Phi(\varrho)=\Pi\circ \pi_{\bar{r}}\circ \Phi\left(\theta+\sum_{d=0}^{\bar{r}+2}\varrho_d\right)=\Pi\circ \pi_{\bar{r}}\circ \Phi(\varrho_{\bar{d}})\neq 0,
$$
which is a contradiction and proves that $\varrho$ is a homogeneous polynomial of degree $\delta+2$.

In summary, we have proved so far that
for any $d\in 2\N$, the restriction of $\Phi$ to $\mathcal{V}_d$ is an isomorphism to the $\Q$ vector space $\overline{\rm Val}_{(d-2)}$ of $(d-2)$-dilative simple
$\mathcal G(\Z^2)$ equivariant
valuations 
$Z:\mathcal{P}(\Z^2)\to \Q[[x,y]]$. 
This completes the proof of the first case in 
Theorem~\ref{thm:dilative_parameters}.
\end{proof}

The proofs of the other two cases in Theorem \ref{thm:dilative_parameters} are shorter, since we can make use of the fact that the respective spaces of $\delta$-dilative valuations $\overline{\mathrm{Val}}_\delta$ are 1-dimensional.

\begin{proof}[Proof of Theorem \ref{thm:dilative_parameters}, Case 2]
Let $\delta > -2$ be an odd integer. We start by showing that the valuation $Z$ obtained from Theorem \ref{PowerSeriesValGL2Z} by choosing $c=0$, $g(x)=x^{\delta/2}\sinh(\tfrac{\sqrt x}{2})$, $f_2=0$ is $\delta$-dilative. Note that $Z$ vanishes on points, since $c=0$. We show that $Z$ is $\delta$-dilative for the standard segment, i.e.,
\begin{equation}
\label{eq:dilative_segments}
    Z([0,me_1])(x,y) = m^{-\delta}Z([0,e_1])(mx,my)
\end{equation}
holds for all $m\in\N$. We write
$[0,me_1] = \bigcup_{0\leq k < m} ke_i + [0,e_1] $,
where the union is disjoint up to points. Thus, by Theorem \ref{PowerSeriesValGL2Z}, we have
\[\begin{split}
Z([0,me_1]) &= \left(\sum_{k=0}^{m-1} \exp(kx)\right)Z([0,e_1]) = \frac{1-\exp(mx)}{1-\exp(x)} x^\delta\sinh(\tfrac x2)\exp(\tfrac x2)\\
&=x^\delta \,\frac{\exp(mx)-1}{2}.
\end{split}
\]
Regarding the right hand side of \eqref{eq:dilative_segments}, we have
\[
m^{-\delta}Z([0,e_1])(mx,my) = m^{-\delta}(mx)^\delta\sinh(\tfrac{mx}{2}) \exp(\tfrac{mx}{2}) =x^\delta \,\frac{\exp(mx)-1}{2},
\]
so \eqref{eq:dilative_segments} is verified. It follows from the affine equivariance that $Z$ is $\delta$-dilative for all segments.

In order to see that $Z$ is $\delta$-dilative for 2-dimensional polygons, we observe that
\(
Z(P) = \tfrac 12\sum_{e\subset P} Z(e)
\)
holds for any proper polygon, where the sum ranges over the edges of $P$: For a unimodular triangle, this follows from \eqref{ZZ2Z1} together with the $\mathcal{G}(\Z^2)$ equivariance. For an arbitrary polygon, we see that the interior edges of a unimodular triangulation $\mathcal T$ of $P$ do not contribute to the sum
\begin{equation}
\label{eq:Ztriang}
    Z(P) = \sum_{S} \sum_{e\subset S} \tfrac 12 Z(e) - \sum_{e'} Z(e'),
\end{equation}
where $S$ ranges over the unimodular triangles in $\mathcal{T}$, $e$ over the edges of $S$ for a given $S$, and $e'$ over the edges in $\mathcal{T}$. Hence, the only contributions to \eqref{eq:Ztriang} come from the boundary edges of $\mathcal T$ which sum up to the edges of $P$. This, together with the $\delta$-dilative property for segments, proves that $Z$ is indeed a $\delta$-dilative valuation.

The reverse implication now follows from the fact that $\Val_\delta$ is a 1-dimensional space of non-simple valuations for odd $\delta>-2$. 
\end{proof}

\begin{proof}[Proof of Theorem~\ref{thm:dilative_parameters}, Case 3]
As in the previous proof, it suffices to show that the valuation $Z$ from Theorem \ref{PowerSeriesValGL2Z} with $c=1$, $g(x) = \cosh\left(\tfrac {\sqrt{x}}{2}\right)$ and $f_2 = 0$ is 0-dilative. For a point $p=(p_1,p_2)\in\Z^2$ we have
\[
Z(p) = Z(\{0\} + p) = \exp({p_1x + p_2y})Z(\{0\}) = \exp({p_1x + p_2y})
\]
and for $m\in\N$ it follows that
\[\begin{split}
Z(mp)(x,y) &= Z(p + (m-1)p)(x,y)\\
&= \exp((m-1)(p_1x + p_2y))\exp({p_1x+p_2y}) \\
&= \exp({p_1mx +p_2my}) = Z(p)(mx,my).
\end{split}\]
Hence, $Z$ is 0-dilative for points. Next, we show that $Z$ is 0-dilative for segments, i.e., we aim to verify \eqref{eq:dilative_segments} for $\delta=0$. This time, we have to take into account that $Z$ does not vanish on points, so we compute
\[
\begin{split}
    Z([0,me_1])(x,y) &= \left(\sum_{k=0}^{m-1}\exp(kx)\right) Z([0,e_1]) - \sum_{\ell=1}^{m-1}\exp(\ell x) Z(\{0\})\\
    &= \frac{1-\exp(mx)}{1-\exp(x)} \cosh(\tfrac x2)\exp(\tfrac x2) - \frac{\exp(x)-\exp(mx)}{1-\exp(x)}\\
    & = \tfrac12 \frac{(1-\exp(mx))(\exp(x) + 1) - 2(\exp(x) - \exp(mx))}{1-\exp(x)}\\
    & = \frac 12 \frac{1-\exp(x)+\exp(mx) -\exp(mx + 1)}{1-\exp(x)}\\
    & = \tfrac 12 \frac{(1-\exp(x))(1+\exp(mx)}{1-\exp(x)}\\
    & = \frac {\exp(\tfrac{mx}{2}) + \exp(-\tfrac{mx}{2})}{2}\exp(\tfrac{mx}{2})\\
    &= \cosh(\tfrac {(\sqrt{mx})^2}{2})\exp(\tfrac{mx}{2}) = Z([0,e_1])(mx,my).
\end{split}
\]

This shows that $Z$ is also 0-dilative on segments. In order to see that $Z$ is 0-dilative on full-dimensional polygons, we can proceed as in the previous proof, exploiting the fact that interior lattice points of a polygon are contained in as many triangles as edges for any unimodular triangulation. Hence, the formular for $Z(P)$ is the same as in \eqref{eq:Ztriang} and we see that 
$Z$ is indeed 0-dilative.

As in Case 2, the reverse implication follows from $\dim\Val_0 = 1$ (cf.\ Theorem~\ref{theorem:dilative_decomp}).
\end{proof}


As a by-product of the proofs we obtain that if $\delta$ is odd or zero, any $\delta$-dilative valuation acts as an equivariant surface area measure.

\begin{coro}
\label{coro:surface}
    Let $Z\colon\mathcal{P}(\Z^2) \to \Q[x,y]$ be a $\delta$-dilative $\mathcal{G}(\Z^2)$ equivariant valuation, where $\delta>-2$ is odd or zero. Then, $Z(P) = \tfrac 12\sum_{e\subset P} Z(e)$, where $e$ ranges over the edges of $P$.
\end{coro}

The conceptual difference of Case 1 to Cases 2 and 3 is that in the latter cases we knew a priori that the space of $\delta$-dilative valuations $\Val_\delta$ is of the same dimension than the valuations described by the parameters in the respective case, whereas in the first case it was \emph{not} a priori clear that the dimensions of $\Val_\delta$ and $\mathcal{V}_{\delta+2}$ agree. A posteriori we see that they indeed agree and together with Theorem \ref{theorem:dilative_decomp} we obtain:

\begin{coro}
\label{coro:vd_dims}
Let $d\in\N$ and let $\mathcal V_d$ be the vector space of $d$-homogeneous polynomials that satisfy \eqref{rhoformula}. Then,
\[
\dim \mathcal V_d = \begin{cases}
0, & d\text{ odd},\\
                \lfloor \tfrac{d}{12}\rfloor + 1, &d \text{ even and }\,\delta\,{\rm mod}\,12 \neq 2,\\
                \lfloor \tfrac{d}{12}\rfloor,&d \text{ even and }\,\delta\,{\rm mod}\,12 = 2.
\end{cases}
\]
\end{coro}

\bigskip

\noindent{\bf Acknowledgements.} We are grateful to Monika Ludwig  (TU Wien), Gergely Harcos (R\'enyi Institute) and \'Arp\'ad T\'oth (R\'enyi Institute) for enlightening discussions.
We are also thankful for the NKKP ADVANCED grant 150613 supporting K.J.\ B\"or\"oczky, for the NKFIH grant 138828 supporting M.\ Domokos, for the Deutsche Forschungsgemeinschaft (DFG) project number 539867386 supporting A.\ Freyer, and  for the National Natural Science Foundation of China 12201388 and the Austrian Science Fund (FWF) I3027 supporting J.\ Li.

K\'aroly J. B\"or\"oczky, HUN-REN Alfr\'ed R\'enyi Institute of Mathematics, boroczky.karoly.j@renyi.hu\\

M\'aty\'as Domokos, HUN-REN Alfr\'ed R\'enyi Institute of Mathematics, domokos.matyas@renyi.hu\\

Ansgar Freyer, FU Berlin, Fachbereich Mathematik und Informatik, Arnimallee 2, 14195 Berlin, a.freyer@fu-berlin.de\\

Christoph Haberl, Vienna University of Technology, Institute of Discrete Mathematics and Geometry, Wiedner Hauptstraße 8-10/104, 1040 Vienna, Austria, christoph.haberl@tuwien.ac.at\\

Jin Li, Department of Mathematics, and Newtouch Center for Mathematics, Shanghai University, li.jin.math@outlook.com 

\end{document}